\documentclass[11pt]{amsart}

\usepackage[T1]{fontenc}
\usepackage[utf8]{inputenc}
\usepackage{lmodern}
\usepackage{microtype}
\usepackage{amsmath,amssymb,mathtools}
\usepackage{amsthm}
\usepackage{enumitem}
\usepackage{booktabs}
\usepackage{needspace}
\usepackage{tikz}
\usetikzlibrary{matrix}
\usetikzlibrary{positioning}
\usepackage[colorlinks=true,linkcolor=blue!55!black,
  citecolor=green!35!black,urlcolor=blue!60!black]{hyperref}
\setlist[enumerate]{leftmargin=*,itemsep=2pt,topsep=4pt}
\setlist[itemize]{leftmargin=*,itemsep=2pt,topsep=4pt}

\newtheorem{theorem}{Theorem}[section]
\newtheorem{proposition}[theorem]{Proposition}
\newtheorem{lemma}[theorem]{Lemma}
\newtheorem{corollary}[theorem]{Corollary}
\newtheorem{conjecture}[theorem]{Conjecture}
\newcommand{\repeatedtheoremname}{}
\newtheorem*{repeatedtheoreminner}{\repeatedtheoremname}
\newenvironment{repeatedtheorem}[1]
  {\renewcommand{\repeatedtheoremname}{Theorem~\ref*{#1}}%
   \begin{repeatedtheoreminner}}
  {\end{repeatedtheoreminner}}
\theoremstyle{definition}
\newtheorem{definition}[theorem]{Definition}
\newtheorem{example}[theorem]{Example}
\newtheorem{question}[theorem]{Question}
\theoremstyle{remark}
\newtheorem{remark}[theorem]{Remark}

\newcommand{\catw}{\operatorname{cat}_{w}}
\newcommand{\cats}{\operatorname{cat}_{s}}
\newcommand{\catLS}{\operatorname{cat}}
\newcommand{\scat}{\operatorname{scat}}
\newcommand{\sd}{\operatorname{sd}}
\newcommand{\OO}{\mathcal O}
\newcommand{\whU}{\widehat U}
\newcommand{\whF}{\widehat F}
\newcommand{\sq}{\mathbin{\sqcup}}

\title[Lusternik-Schnirelmann category of finite spaces]
{Weak, stable, and ordinary Lusternik-Schnirelmann category of finite spaces}

\author[D. Mosquera-Lois]{David Mosquera-Lois}
\address{Department of Mathematics, University of Santiago de Compostela, Spain}
\email{david.mosquera.lois@usc.gal}
\author[K. Tanaka]{Kohei Tanaka}
\address{Academic Assembly, School of Humanities and Social Sciences,
Institute of Social Sciences, Shinshu University,
3-1-1, Matsumoto, Nagano 390-8621, Japan}
\email{tanaka@shinshu-u.ac.jp}

\thanks{The first author has been supported by ED431C 2023/31 (Xunta de Galicia).}

\newcommand{\MainTheoremStatement}{%
Let $P$ be a weakly contractible, noncontractible finite $T_0$-space with $P^{\sharp}$ points.
Suppose that $P\setminus\{g\}$ is nonempty and contractible for some $g\in P$.
For every integer $m\geq2$, let $D_m$ be the discrete $m$-point space and set
$P_m=P\oplus D_m$.  Then $P_m$ is connected, has $P^{\sharp}+m$ points, and satisfies
\[
 \bigl(\catw(P_m),\cats(P_m),\catLS(P_m)\bigr)=(1,2,m).
\]
Moreover, $\catLS(\sd^kP_m)=2$ for every $k\geq1$.
}

\newcommand{\RealizationTheoremStatement}{%
The set $\mathfrak R$ contains the following triples:
\begin{enumerate}
 \item every $(a,b,c)$ with $1\leq a<b\leq2a$ and $c\geq b$;
 \item every $(a,a,c)$ with $2\leq a\leq c$;
 \item the triple $(1,1,1)$.
\end{enumerate}
Moreover, a triple of the form $(1,1,c)$ is realizable only when $c=1$.
}

\begin{document}

\begin{abstract}
The weak, stable, and ordinary Lusternik--Schnirelmann categories of a finite
$T_0$-space satisfy $\catw(X)\leq\cats(X)\leq\catLS(X)$.
We give a general construction answering the simultaneous-strictness question
of C\'ardenas, Flores, Quintero, and Villar-Li\~n\'an.
If $P$ is weakly contractible but noncontractible and deleting one point
makes $P$ contractible, then adjoining $m\geq2$ incomparable maxima gives a
connected space with category triple $(1,2,m)$, whose ordinary category is
two after every positive number of subdivisions.  Applying the construction
to a nine-point space gives examples on $m+9$ points and simultaneous
strictness on twelve points.  By additivity under disjoint unions, we also
realize every triple $(a,b,c)$ with $1\leq a<b\leq2a$ and $c\geq b$, and every
triple $(a,a,c)$ with $2\leq a\leq c$.  We conclude with questions on connected
realization and the minimum size of a strict example.
\end{abstract}

\maketitle

\section{Introduction}\label{sec:introduction}

The Lusternik--Schnirelmann category of a topological space $X$, denoted by
$\catLS(X)$, is the least number of open subsets covering $X$ whose
inclusions into $X$ are nullhomotopic; it is infinite if no finite such
cover exists.  Thus category measures the number of pieces needed to cover
a space when each piece can be contracted within the ambient space.
It is a homotopy invariant with a classical role in the study of critical
points and in algebraic topology.  We refer to the monograph of Cornea,
Lupton, Oprea, and Tanr\'e \cite{CLOT} for the general theory.  Throughout
this paper we use the unreduced convention, in which a nonempty
contractible space has category one.

Combinatorial approaches to this invariant make it possible to study
categorical covers through finite structures.  Fern\'andez-Ternero,
Mac\'ias-Virg\'os, and Vilches \cite{FernandezMaciasVilches2015} introduced
a simplicial LS-category using covers by subcomplexes and the contiguity
relation between simplicial maps.  Their construction is an invariant of
strong homotopy type and provides comparisons with the ordinary category
of finite topological spaces through order complexes and face posets.
The theory was developed further by Fern\'andez-Ternero, Mac\'ias-Virg\'os,
Minuz, and Vilches \cite{FernandezMaciasMinuzVilches2019}, who studied
simplicial category and its relation to the category of geometric
realizations.  These works provide a framework for asking which features
of LS-category are retained when one passes between a finite combinatorial
model and its associated polyhedron.

Finite $T_0$-spaces are particularly well suited to this question.  They
can be identified with finite posets, with continuous maps corresponding
to order-preserving maps.  To such a space $X$ one associates its order
complex $\OO(X)$, whose simplices are the nonempty chains of the poset.
McCord's theorem gives a weak homotopy equivalence between $X$ and the
polyhedron $|\OO(X)|$ \cite{McCord1966}.  Nevertheless, $X$ and
$|\OO(X)|$ need not have the same homotopy type.  For example, a finite
space can be noncontractible while its order complex has contractible
realization.  Consequently, the passage to a polyhedron can lose
information relevant to categorical covers of the finite space.

C\'ardenas, Flores, Quintero, and Villar-Li\~n\'an
\cite{CardenasFloresQuinteroVillar2025} studied covering invariants of
finite spaces by combining homotopy-theoretic methods with graph and
hypergraph techniques.  Building on the simplicial comparisons, they
introduced the weak and stable LS-categories
\[
 \catw(X)=\catLS(|\OO(X)|),\qquad
 \cats(X)=\min_{k\geq0}\catLS(\sd^kX),
\]
where $\sd X$ is the face poset of $\OO(X)$.  The minimum in the second
definition is the eventual value of the nonincreasing sequence of
categories of the iterated subdivisions.  These invariants satisfy
\begin{equation}\label{eq:basic-chain}
 \catw(X)\leq\cats(X)\leq\catLS(X);
\end{equation}
see \cite[Definition~2.1 and equation~(2.1)]{CardenasFloresQuinteroVillar2025}.

The two comparisons in \eqref{eq:basic-chain} express different aspects of
the relationship between finite spaces and polyhedra.  A strict inequality
$\cats(X)<\catLS(X)$ means that subdivision allows more efficient
categorical covers than the original finite topology.  On the other hand,
$\catw(X)<\cats(X)$ means that no number of finite subdivisions suffices
to attain the category of the associated polyhedron.  Studying both gaps
therefore asks how much categorical information subdivision can remove,
and how much can persist even after arbitrarily many subdivisions.

The examples recalled in
\cite[Remark~2.3]{CardenasFloresQuinteroVillar2025} exhibit each gap
separately.  The height-one spaces with Hasse graph $K_{2,n}$ have category
triple $(2,2,n)$, while a weakly contractible noncontractible example has
triple $(1,2,2)$.  The same authors ask whether there is a finite $T_0$-space
on which both inequalities in \eqref{eq:basic-chain} are strict
\cite[p.~209]{CardenasFloresQuinteroVillar2025}.

There is a useful distinction between the disconnected and connected
versions of this problem.  The three categories are additive under
disjoint unions, as we recall in Proposition~\ref{prop:additivity}.
One can therefore combine examples exhibiting the separate gaps by placing
them in different components.  A connected example requires both phenomena
to occur within the same component.  It is then natural to ask a stronger
question: can the weak and stable categories be kept fixed while the
ordinary category grows without bound among connected spaces?

We answer this question by a construction that applies to a class of
weakly contractible noncontractible spaces.  The additional hypothesis is
that deleting a single point leaves a contractible finite space.  It is this
deletion property that produces an efficient categorical cover after
subdivision.

\begin{theorem}\label{thm:main}
\MainTheoremStatement
\end{theorem}

The hypotheses hold for a nine-point space $P(9)$ from the classification of
Cianci and Ottina \cite{CianciOttina2020}.  Thus
$X_m=P(9)\oplus D_m$ has $m+9$ points and category triple $(1,2,m)$.
In particular,
\[
 1=\catw(X_3)<\cats(X_3)=2<\catLS(X_3)=3,
\]
so simultaneous strictness occurs on a connected twelve-point space.
More generally, ordinary category is unbounded among connected weakly
contractible spaces of stable category two.  For $m>2$, the ordinary
category drops from $m$ to two at the first subdivision; the remaining gap
between stable and weak category persists through every further subdivision.

The three parts of the argument use distinct features of the base $P$.
Its contractible order complex controls weak category.  The deletion
$P\setminus\{g\}$ and the closed star of $g$ give a cover by two strongly
collapsible subcomplexes after adjoining the maxima.  Finally, the
noncontractibility of $P$, together with a retraction argument, prevents a
categorical open set from containing two new maxima.  This last argument
computes the ordinary category without any weak-contractibility assumption.

We call a point \emph{essential} when its deletion leaves a nonempty
contractible finite space.  If it is also a weak point in the sense of
Barmak and Minian \cite{BarmakMinian2008}, then the base is automatically
weakly contractible.  Their Wallet example (see Example \ref{ex:wallet}) already exhibits this combination
of properties.  The terminology isolates the hypothesis used here;
Section~\ref{sec:core} applies the criterion to $P(9)$, its opposite poset,
six ten-point bases, and the Wallet.  A seventh ten-point base yields the
same category triple by a separate simplicial argument in
Example~\ref{ex:ten-point-bases}.  The nine-point examples have the smallest
possible base size under the hypotheses of Theorem~\ref{thm:main}.

Once such a family is available, the existence question leads to a
realization problem for triples.  Let
\[
 \mathfrak R=\left\{(\catw(X),\cats(X),\catLS(X)):
 X\text{ is a nonempty finite }T_0\text{-space}\right\}.
\]
The inequalities \eqref{eq:basic-chain} give necessary numerical
restrictions, but do not by themselves describe $\mathfrak R$.
Our connected family supplies building blocks which, together with
additivity and the known height-one examples, yield explicit
representatives throughout the following regions.

\begin{theorem}\label{thm:realization}
\RealizationTheoremStatement
\end{theorem}

The representatives in the first region have $c+9(b-a)$ points and are
generally disconnected.  The boundary $b\leq2a$ records the reach of the
additive construction; we do not know whether it is a universal restriction.
Connected realization beyond the displayed families is also open.  The
small size of $X_3$ raises a separate question of efficiency, and we
conjecture that twelve points is the minimum for simultaneous strictness.

For related approaches via small categories, relation matrices, and sectional
category, see Tanaka \cite{Tanaka2018,Tanaka2020,Tanaka2024}.

Section~\ref{sec:preliminaries} recalls the necessary comparisons.
Section~\ref{sec:family} proves the general construction in
Theorem~\ref{thm:main}, and Section~\ref{sec:core} gives the small bases
and the twelve-point example.  Section~\ref{sec:realization} establishes
Theorem~\ref{thm:realization}, and Section~\ref{sec:questions} discusses
further questions.

\section{Finite spaces and three categories}\label{sec:preliminaries}

Finite spaces are assumed nonempty and $T_0$ unless an empty summand is
explicitly allowed; all simplicial complexes are finite.
References for classical LS-category and finite-space homotopy theory include
\cite{CLOT} and \cite{Barmak2011,McCord1966,Stong1966}; the stable and
simplicial comparison results used here are developed in
\cite{FernandezMaciasVilches2015,FernandezMaciasMinuzVilches2019}.

Every finite $T_0$-space is identified with its associated poset (see \cite{Barmak2011});
continuous maps are precisely the order-preserving maps.  We use
the convention
\[
  U_x=\{y:y\leq x\},\qquad F_x=\{y:y\geq x\}.
\]
Thus open subsets are lower sets.  We put
the strict lower set $\whU_x=U_x\setminus\{x\}$ and
the strict upper set $\whF_x=F_x\setminus\{x\}$.  The order complex $\OO(X)$ has the points of
$X$ as vertices and the nonempty chains of $X$ as simplices.

A point $x$ is an \emph{up beat point} if $\whF_x$ has a minimum and a
\emph{down beat point} if $\whU_x$ has a maximum.  Deleting a beat point is a
strong deformation retract.  A finite $T_0$-space is contractible if and only
if successive beat-point deletions reduce it to one point
\cite{Stong1966}.  Successively deleting beat points until none remain
produces a \emph{core}.  A finite space without beat points is also called
\emph{minimal}; this does not assert minimum cardinality within a weak
homotopy type.

A point $g\in P$ is a \emph{weak point} if $\whU_g$ or $\whF_g$ is
contractible.  Deleting a weak point induces a weak homotopy equivalence
\cite[Definition~3.2 and Proposition~3.3]{BarmakMinian2008}. We introduce two new key notions.

\begin{definition}\label{def:essential}
A point $g$ of a finite $T_0$-space $P$ is \emph{essential} if
$P\setminus\{g\}$ is nonempty and contractible.  It is an
\emph{essential weak point} if it is both essential and a weak point.
\end{definition}

Equivalently, deleting an essential point leaves a space that reduces to
one point by beat-point deletions.  The weak-point condition concerns a
strict lower or upper set, whereas essentiality concerns the whole
complement.  If $g$ is an essential weak point, then $P$ is weakly
contractible.  Essentiality alone does not imply this: in the four-point
poset with two minima below two maxima, deleting any point leaves a
contractible space, but the order complex is a circle.

\begin{definition}
An open subset $U\subseteq X$ is \emph{categorical in $X$} if the inclusion
$U\hookrightarrow X$ is nullhomotopic.  The ordinary category $\catLS(X)$ is
the least number of categorical open subsets that cover $X$.
\end{definition}

The barycentric subdivision $\sd X$ is the face poset of $\OO(X)$, and
$\sd^kX$ denotes its $k$-fold iterate.  The sequence
$\catLS(\sd^kX)$ is nonincreasing and therefore eventually constant.  We set
\[
  \cats(X)=\min_{k\geq0}\catLS(\sd^kX),
  \qquad
  \catw(X)=\catLS\bigl(|\OO(X)|\bigr).
\]
The comparison \eqref{eq:basic-chain} follows from the finite-space and
simplicial comparison results of
\cite{FernandezMaciasVilches2015,FernandezMaciasMinuzVilches2019}.

We also use simplicial category in the unreduced convention.  Two simplicial
maps $f,g\colon L\to K$ are contiguous if $f(\sigma)\cup g(\sigma)$ is a
simplex of $K$ for every simplex $\sigma$ of $L$.  A subcomplex $L\subseteq K$
is categorical if its inclusion lies in the contiguity class of a constant
map, where contiguity class means equivalence under a finite sequence of
contiguous maps.  The simplicial category $\scat(K)$ is the least number of
categorical subcomplexes covering $K$.  Thus our values are one greater than
the reduced values used in \cite{FernandezMaciasVilches2015}.

A vertex $v$ of $K$ is dominated by $w\ne v$ if every maximal simplex
containing $v$ also contains $w$.  A sequence of dominated-vertex deletions
is a strong collapse; $K$ is strongly collapsible if such a sequence reduces
it to a vertex.  A strongly collapsible subcomplex is categorical in any
ambient complex, and every simplicial cone is strongly collapsible.  We write
$\operatorname{st}_K(v)$ for the closed star of a vertex $v$ in $K$.
The finite-space comparison results give
\begin{equation}\label{eq:comparison-scat}
 \catLS(\sd X)\leq\scat(\OO(X))\leq\catLS(X)
\end{equation}
and
\begin{equation}\label{eq:contractibility-detection}
 \begin{aligned}
 X\text{ contractible}
 &\ \Longleftrightarrow\ \OO(X)\text{ strongly collapsible}\\
 &\ \Longleftrightarrow\ \sd X\text{ contractible};
 \end{aligned}
\end{equation}
see \cite[Section~6]{FernandezMaciasVilches2015}.  In particular, no iterated
subdivision of a noncontractible finite space is contractible.

If $P$ and $Q$ are posets, their ordinal sum $P\oplus Q$ is the disjoint
union in which every point of $P$ is declared smaller than every point of
$Q$.  Order complexes turn ordinal sums into simplicial joins:
\begin{equation}\label{eq:ordinal-join}
  \OO(P\oplus Q)=\OO(P)*\OO(Q).
\end{equation}

\section{A general construction}\label{sec:family}

Let $P$ be a nonempty finite $T_0$-space and let
$D_m=\{t_1,\dots,t_m\}$ be the discrete $m$-point poset, where $m\geq2$.
Write
\begin{equation}\label{eq:Pm}
 P_m=P\oplus D_m.
\end{equation}
The $t_i$ are incomparable maxima above every point of $P$.  The space
$P_m$ is connected, even when $P$ is disconnected, and has $P^{\sharp}+m$ points.
We first compute ordinary category, then identify separate conditions
controlling the weak and stable categories.

\begin{lemma}\label{lem:Xm-noncontractible}
If $P$ is noncontractible, then $P\oplus J$ is noncontractible for every
finite discrete space $J$ with $J^{\sharp} \geq2$.  If, in addition, $P$ has no
beat points, then neither does $P\oplus J$.
\end{lemma}

\begin{proof}
Every beat-point deletion in $P$ remains a beat-point deletion after
adjoining $J$.  Indeed, strict lower sets of old points are unchanged.
If the strict upper set of an old point has a minimum, that minimum is
still below all the new points.  Hence a reduction of $P$ to a core $C$
extends to a reduction of $P\oplus J$ to $C\oplus J$.

Since $P$ is noncontractible, $C$ has at least two points.  It has no
greatest point, for a finite space with a greatest point is contractible.
We check that $C\oplus J$ has no beat points.  An old point cannot become a
down beat point because its strict lower set is unchanged.  For a
nonmaximal old point, a minimum of its new strict upper set would already
be a minimum among its old strict upper points.  An old maximal point has
strict upper set $J$, which has no minimum.  Finally, a new point has
strict lower set $C$, which has no maximum, and empty strict upper set.
Thus $C\oplus J$ is a space with more than one point and no beat points,
so it is noncontractible.  The same follows for $P\oplus J$ by homotopy
invariance.  When $P$ itself has no beat points, the argument applies
with $C=P$.
\end{proof}

\begin{proposition}\label{prop:ordinary}
If $P$ is noncontractible, then $\catLS(P_m)=m$ for every $m\geq2$.
\end{proposition}

\begin{proof}
The $m$ principal open sets $U_{t_i}=P\cup\{t_i\}$ have maxima and cover
$P_m$, so $\catLS(P_m)\leq m$.

For the reverse inequality, suppose that a categorical open subset $U$
contains at least two new maxima.  Put $J=U\cap D_m$.  Since $U$ is a
lower set, $U=P\cup J=:A_J$, where $J^{\sharp}\geq2$.
Fix $t\in J$.  The map $r\colon P_m\to A_J$ fixing $A_J$ and sending
each point of $D_m\setminus J$ to $t$ is an order-preserving retraction.
If the inclusion $A_J\hookrightarrow P_m$ were nullhomotopic, composing
with $r$ would make the identity of $A_J$ nullhomotopic, contrary to
Lemma~\ref{lem:Xm-noncontractible}.  Thus a categorical open set contains
at most one new maximum, and every categorical cover has at least $m$
members.
\end{proof}

\begin{samepage}
\begin{proposition}\label{prop:weak}
If $P$ is weakly contractible, then $\catw(P_m)=1$ for every $m\geq2$.
\end{proposition}

\begin{proof}
McCord's theorem and Whitehead's theorem imply that $|\OO(P)|$ is
contractible.  By \eqref{eq:ordinal-join},
$\OO(P_m)=\OO(P)*\OO(D_m)$, whose realization is therefore contractible.
\end{proof}

\end{samepage}

\begin{proposition}\label{prop:stable}
Suppose that $P$ is noncontractible and has an essential point $g$.
Then, for every $m\geq2$ and $k\geq1$,
\[
 \catLS(\sd^kP_m)=\cats(P_m)=2.
\]
\end{proposition}

\begin{proof}
Put $K=\OO(P_m)$ and consider the subcomplexes
\[
 K_0=\OO(P_m\setminus\{g\}),\qquad K_1=\operatorname{st}_K(g).
\]
Every simplex either avoids $g$ or belongs to its closed star, so these
subcomplexes cover $K$.  The second is a cone.  The first is
\[
 K_0=\OO(P\setminus\{g\})*\OO(D_m).
\]
By \eqref{eq:contractibility-detection}, $\OO(P\setminus\{g\})$ is strongly
collapsible.  A dominated-vertex deletion remains such a deletion after
taking a join: each maximal simplex of a join is the union of maximal
simplices of its factors.  Hence $K_0$ strongly collapses to
$\{v\}*\OO(D_m)$ for some vertex $v$, and this last complex is a cone.
Both pieces are therefore strongly collapsible, so $\scat(K)\leq2$.
By \eqref{eq:comparison-scat}, $\catLS(\sd P_m)\leq2$.

Lemma~\ref{lem:Xm-noncontractible} shows that $P_m$ is noncontractible.
By \eqref{eq:contractibility-detection}, no subdivision of $P_m$ is
contractible.  Monotonicity under subdivision now gives
\[
 2\leq\catLS(\sd^kP_m)\leq\catLS(\sd P_m)\leq2
 \qquad(k\geq1),
\]
and the stable-category equality follows.
\end{proof}

\begin{repeatedtheorem}{thm:main}
\MainTheoremStatement
\end{repeatedtheorem}

\begin{proof}[Proof of Theorem~\ref{thm:main}]
The hypotheses give an essential point of $P$.
Apply Propositions~\ref{prop:ordinary}, \ref{prop:weak}, and
\ref{prop:stable}, and the connectedness and cardinality observations
following \eqref{eq:Pm}.
\end{proof}

\begin{corollary}\label{cor:essential-weak}
Let $P$ be a noncontractible finite $T_0$-space with an essential weak
point.  Then $P_m=P\oplus D_m$ satisfies all the conclusions of
Theorem~\ref{thm:main}.  In particular, this applies to every finite space
without beat points that has an essential weak point.
\end{corollary}

\begin{proof}
Deleting the given weak point yields a weak homotopy equivalence to a
contractible space, so $P$ is weakly contractible.  Theorem~\ref{thm:main}
applies.  A space without beat points and with an essential point has at
least two points and is therefore noncontractible.
\end{proof}

\section{Small bases and examples}\label{sec:core}

Cianci and Ottina proved that, up to homeomorphism, there are exactly two
nine-point weakly contractible noncontractible finite spaces
\cite{CianciOttina2020}.  We first describe a representative of one of those
two types and verify the hypotheses of the general criterion.  To make the construction self-contained, let $P(9)$ be
the poset with minima $c_0,c_1,c_2$, middle points $p,q,y$, maxima $g,h,k$,
and covering relations
\begin{align*}
 c_0&<p,y, & c_1&<p,q,y, & c_2&<q,y,\\
 p&<g,h, & q&<g,k, & y&<h,k.
\end{align*}
Its Hasse diagram is shown in Figure~\ref{fig:P9}.

\begin{figure}[htbp]
\centering
\begin{tikzpicture}[x=0.8cm,y=0.85cm,
  every node/.style={font=\scriptsize,inner sep=1.5pt,fill=white}]

  \node (c0) at (0.0,0) {$c_0$};
  \node (c1) at (1.5,0) {$c_1$};
  \node (c2) at (3.0,0) {$c_2$};

  \node (p) at (0.0,1) {$p$};
  \node (y) at (1.5,1) {$y$};
  \node (q) at (3.0,1) {$q$};

  \node (h) at (0.0,2) {$h$};
  \node (g) at (1.5,2) {$g$};
  \node (k) at (3.0,2) {$k$};

  \draw (c0)--(p);
  \draw (c0)--(y);
  \draw (c1)--(p);
  \draw (c1)--(q);
  \draw (c1)--(y);
  \draw (c2)--(q);
  \draw (c2)--(y);

  \draw (p)--(g);
  \draw (p)--(h);
  \draw (q)--(g);
  \draw (q)--(k);
  \draw (y)--(h);
  \draw (y)--(k);

\end{tikzpicture}
\caption{The nine-point core $P(9)$.}
\label{fig:P9}
\end{figure}
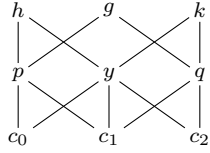

\begin{lemma}\label{lem:P9}
The space $P(9)$ has no beat points, and $g$ is an essential weak point.
In particular, $P(9)$ is weakly contractible but noncontractible.
\end{lemma}

\begin{proof}
The minimal elements of $\whF_{c_0},\whF_{c_1},\whF_{c_2}$ are,
respectively,
\[
  \{p,y\},\qquad \{p,q,y\},\qquad \{q,y\}.
\]
For each middle point, both the maximal elements of its strict lower set and
the minimal elements of its strict upper set contain at least two points.
Finally, the maximal elements of $\widehat{U}_g, \widehat{U}_h, \widehat{U}_k$  are
\[
  \{p,q\},\qquad \{p,y\},\qquad \{q,y\}.
\]
Thus $P(9)$ has no beat point.  Since it has more than one point, it is
noncontractible.

After deleting $g$, the following is a valid beat-point reduction to the
single point $y$:
\[
  p,\ q,\ c_0,\ c_2,\ h,\ k,\ c_1.
\]
Indeed, in order, $p$ and $q$ are up beat points dominated by $h$ and $k$;
$c_0$ and $c_2$ are up beat points dominated by $y$; $h$ and $k$ are down
beat points dominated by $y$; and $c_1$ is an up beat point dominated by
$y$.  The corresponding vertex deletions strongly collapse
$\OO(P(9)\setminus\{g\})$.

The strict lower set $\whU_g$ has Hasse graph the path
\[
 c_0-p-c_1-q-c_2.
\]
Deleting $c_0,c_2,p,q$, in that order, leaves $c_1$; hence $\whU_g$ is
contractible.  Thus $g$ is both weak and essential.  Deleting this weak
point yields a weak homotopy equivalence from a contractible space to
$P(9)$, so $P(9)$ is weakly contractible.
\end{proof}

\begin{corollary}\label{cor:subdivision-profile}
For every $m\geq2$, the connected space
\begin{equation}\label{eq:Xm}
 X_m=P(9)\oplus D_m
\end{equation}
has $m+9$ points and category triple $(1,2,m)$.  Moreover,
$\catLS(\sd^kX_m)=2$ for every $k\geq1$.  In particular, $X_3$ has twelve
points and satisfies $\catw(X_3)<\cats(X_3)<\catLS(X_3)$.
\end{corollary}

\begin{proof}
Apply Corollary~\ref{cor:essential-weak} to $P(9)$ and $g$.
\end{proof}

\begin{example}[The opposite base]\label{ex:opposite}
The opposite finite space $P(9)^{\mathrm{op}}$ also has no beat points, and deleting
$g$ again leaves a contractible finite space.  The strict upper set of $g$ is now
contractible, so $g$ is again a weak point.  Thus $P(9)^{\mathrm{op}}\oplus D_m$ gives another family with
category triple $(1,2,m)$ on $m+9$ points.  The two bases $P(9)$ and $P(9)^{\mathrm{op}}$ are not isomorphic:
$P(9)$ has a minimal element $c_1$ such that $F_{c_1}^{\sharp}=7$, whereas every
minimal element $x \in \{g,h,k\}$ of $P(9)^{\mathrm{op}}$ has $F_{x}^{\sharp}=6$. Since the
new maxima are precisely the maximal elements of either ordinal sum, an
isomorphism of the sums would restrict to an isomorphism of the bases.
Hence the resulting spaces are not homeomorphic.
\end{example}

\begin{example}[Ten-point bases]\label{ex:ten-point-bases}
Let $P$ be one of the seven ten-point posets of Types I--VII specified in
\cite{DasMawiong2026}, with the incidence relations shown in
Figure~\ref{fig:ten-point-finite-spaces}.  

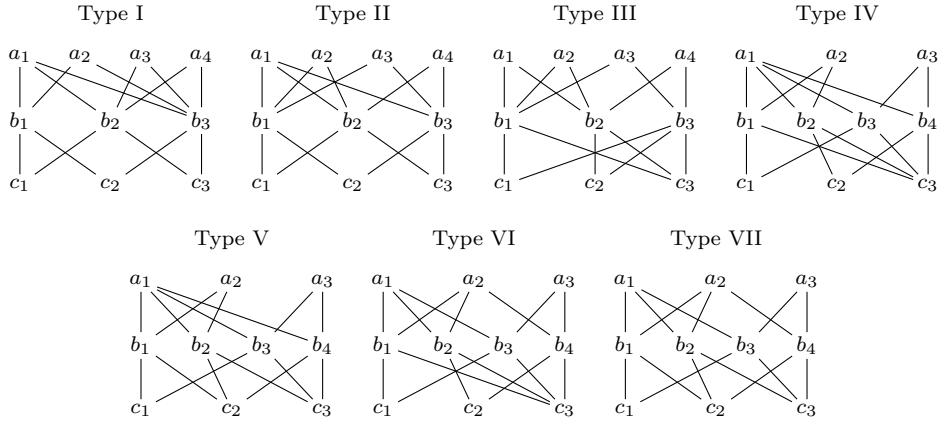
\begin{figure}[htbp]
\centering
\begin{tikzpicture}[x=0.8cm,y=0.85cm,
  every node/.style={font=\scriptsize,inner sep=1.5pt,fill=white}]
\begin{scope}[shift={(0,0)}]
\node at (1.5,2.65) {Type I};
\node (I-a1) at (0.0,2) {$a_1$};
\node (I-a2) at (1.0,2) {$a_2$};
\node (I-a3) at (2.0,2) {$a_3$};
\node (I-a4) at (3.0,2) {$a_4$};
\node (I-b1) at (0.0,1) {$b_1$};
\node (I-b2) at (1.5,1) {$b_2$};
\node (I-b3) at (3.0,1) {$b_3$};
\node (I-c1) at (0.0,0) {$c_1$};
\node (I-c2) at (1.5,0) {$c_2$};
\node (I-c3) at (3.0,0) {$c_3$};
\draw (I-c1)--(I-b1);
\draw (I-c1)--(I-b2);
\draw (I-c2)--(I-b1);
\draw (I-c2)--(I-b3);
\draw (I-c3)--(I-b2);
\draw (I-c3)--(I-b3);
\draw (I-b1)--(I-a1);
\draw (I-b1)--(I-a2);
\draw (I-b2)--(I-a1);
\draw (I-b2)--(I-a3);
\draw (I-b2)--(I-a4);
\draw (I-b3)--(I-a1);
\draw (I-b3)--(I-a2);
\draw (I-b3)--(I-a3);
\draw (I-b3)--(I-a4);
\end{scope}
\begin{scope}[shift={(4,0)}]
\node at (1.5,2.65) {Type II};
\node (II-a1) at (0.0,2) {$a_1$};
\node (II-a2) at (1.0,2) {$a_2$};
\node (II-a3) at (2.0,2) {$a_3$};
\node (II-a4) at (3.0,2) {$a_4$};
\node (II-b1) at (0.0,1) {$b_1$};
\node (II-b2) at (1.5,1) {$b_2$};
\node (II-b3) at (3.0,1) {$b_3$};
\node (II-c1) at (0.0,0) {$c_1$};
\node (II-c2) at (1.5,0) {$c_2$};
\node (II-c3) at (3.0,0) {$c_3$};
\draw (II-c1)--(II-b1);
\draw (II-c1)--(II-b2);
\draw (II-c2)--(II-b1);
\draw (II-c2)--(II-b3);
\draw (II-c3)--(II-b2);
\draw (II-c3)--(II-b3);
\draw (II-b1)--(II-a1);
\draw (II-b1)--(II-a2);
\draw (II-b1)--(II-a3);
\draw (II-b2)--(II-a1);
\draw (II-b2)--(II-a2);
\draw (II-b2)--(II-a4);
\draw (II-b3)--(II-a1);
\draw (II-b3)--(II-a3);
\draw (II-b3)--(II-a4);
\end{scope}
\begin{scope}[shift={(8,0)}]
\node at (1.5,2.65) {Type III};
\node (III-a1) at (0.0,2) {$a_1$};
\node (III-a2) at (1.0,2) {$a_2$};
\node (III-a3) at (2.0,2) {$a_3$};
\node (III-a4) at (3.0,2) {$a_4$};
\node (III-b1) at (0.0,1) {$b_1$};
\node (III-b2) at (1.5,1) {$b_2$};
\node (III-b3) at (3.0,1) {$b_3$};
\node (III-c1) at (0.0,0) {$c_1$};
\node (III-c2) at (1.5,0) {$c_2$};
\node (III-c3) at (3.0,0) {$c_3$};
\draw (III-c1)--(III-b1);
\draw (III-c1)--(III-b3);
\draw (III-c2)--(III-b2);
\draw (III-c2)--(III-b3);
\draw (III-c3)--(III-b1);
\draw (III-c3)--(III-b2);
\draw (III-c3)--(III-b3);
\draw (III-b1)--(III-a1);
\draw (III-b1)--(III-a2);
\draw (III-b1)--(III-a3);
\draw (III-b2)--(III-a1);
\draw (III-b2)--(III-a2);
\draw (III-b2)--(III-a4);
\draw (III-b3)--(III-a3);
\draw (III-b3)--(III-a4);
\end{scope}
\begin{scope}[shift={(12,0)}]
\node at (1.5,2.65) {Type IV};
\node (IV-a1) at (0.0,2) {$a_1$};
\node (IV-a2) at (1.5,2) {$a_2$};
\node (IV-a3) at (3.0,2) {$a_3$};
\node (IV-b1) at (0.0,1) {$b_1$};
\node (IV-b2) at (1.0,1) {$b_2$};
\node (IV-b3) at (2.0,1) {$b_3$};
\node (IV-b4) at (3.0,1) {$b_4$};
\node (IV-c1) at (0.0,0) {$c_1$};
\node (IV-c2) at (1.5,0) {$c_2$};
\node (IV-c3) at (3.0,0) {$c_3$};
\draw (IV-c1)--(IV-b1);
\draw (IV-c1)--(IV-b3);
\draw (IV-c2)--(IV-b2);
\draw (IV-c2)--(IV-b4);
\draw (IV-c3)--(IV-b1);
\draw (IV-c3)--(IV-b2);
\draw (IV-c3)--(IV-b3);
\draw (IV-c3)--(IV-b4);
\draw (IV-b1)--(IV-a1);
\draw (IV-b1)--(IV-a2);
\draw (IV-b2)--(IV-a1);
\draw (IV-b2)--(IV-a2);
\draw (IV-b3)--(IV-a1);
\draw (IV-b3)--(IV-a3);
\draw (IV-b4)--(IV-a1);
\draw (IV-b4)--(IV-a3);
\end{scope}
\begin{scope}[shift={(2,-3.5)}]
\node at (1.5,2.65) {Type V};
\node (V-a1) at (0.0,2) {$a_1$};
\node (V-a2) at (1.5,2) {$a_2$};
\node (V-a3) at (3.0,2) {$a_3$};
\node (V-b1) at (0.0,1) {$b_1$};
\node (V-b2) at (1.0,1) {$b_2$};
\node (V-b3) at (2.0,1) {$b_3$};
\node (V-b4) at (3.0,1) {$b_4$};
\node (V-c1) at (0.0,0) {$c_1$};
\node (V-c2) at (1.5,0) {$c_2$};
\node (V-c3) at (3.0,0) {$c_3$};
\draw (V-c1)--(V-b1);
\draw (V-c1)--(V-b3);
\draw (V-c2)--(V-b1);
\draw (V-c2)--(V-b2);
\draw (V-c2)--(V-b4);
\draw (V-c3)--(V-b2);
\draw (V-c3)--(V-b3);
\draw (V-c3)--(V-b4);
\draw (V-b1)--(V-a1);
\draw (V-b1)--(V-a2);
\draw (V-b2)--(V-a1);
\draw (V-b2)--(V-a2);
\draw (V-b3)--(V-a1);
\draw (V-b3)--(V-a3);
\draw (V-b4)--(V-a1);
\draw (V-b4)--(V-a3);
\end{scope}
\begin{scope}[shift={(6,-3.5)}]
\node at (1.5,2.65) {Type VI};
\node (VI-a1) at (0.0,2) {$a_1$};
\node (VI-a2) at (1.5,2) {$a_2$};
\node (VI-a3) at (3.0,2) {$a_3$};
\node (VI-b1) at (0.0,1) {$b_1$};
\node (VI-b2) at (1.0,1) {$b_2$};
\node (VI-b3) at (2.0,1) {$b_3$};
\node (VI-b4) at (3.0,1) {$b_4$};
\node (VI-c1) at (0.0,0) {$c_1$};
\node (VI-c2) at (1.5,0) {$c_2$};
\node (VI-c3) at (3.0,0) {$c_3$};
\draw (VI-c1)--(VI-b1);
\draw (VI-c1)--(VI-b3);
\draw (VI-c2)--(VI-b2);
\draw (VI-c2)--(VI-b4);
\draw (VI-c3)--(VI-b1);
\draw (VI-c3)--(VI-b2);
\draw (VI-c3)--(VI-b3);
\draw (VI-c3)--(VI-b4);
\draw (VI-b1)--(VI-a1);
\draw (VI-b1)--(VI-a2);
\draw (VI-b2)--(VI-a1);
\draw (VI-b2)--(VI-a2);
\draw (VI-b3)--(VI-a1);
\draw (VI-b3)--(VI-a3);
\draw (VI-b4)--(VI-a2);
\draw (VI-b4)--(VI-a3);
\end{scope}
\begin{scope}[shift={(10,-3.5)}]
\node at (1.5,2.65) {Type VII};
\node (VII-a1) at (0.0,2) {$a_1$};
\node (VII-a2) at (1.5,2) {$a_2$};
\node (VII-a3) at (3.0,2) {$a_3$};
\node (VII-b1) at (0.0,1) {$b_1$};
\node (VII-b2) at (1.0,1) {$b_2$};
\node (VII-b3) at (2.0,1) {$b_3$};
\node (VII-b4) at (3.0,1) {$b_4$};
\node (VII-c1) at (0.0,0) {$c_1$};
\node (VII-c2) at (1.5,0) {$c_2$};
\node (VII-c3) at (3.0,0) {$c_3$};
\draw (VII-c1)--(VII-b1);
\draw (VII-c1)--(VII-b3);
\draw (VII-c2)--(VII-b1);
\draw (VII-c2)--(VII-b2);
\draw (VII-c2)--(VII-b4);
\draw (VII-c3)--(VII-b2);
\draw (VII-c3)--(VII-b3);
\draw (VII-c3)--(VII-b4);
\draw (VII-b1)--(VII-a1);
\draw (VII-b1)--(VII-a2);
\draw (VII-b2)--(VII-a1);
\draw (VII-b2)--(VII-a2);
\draw (VII-b3)--(VII-a1);
\draw (VII-b3)--(VII-a3);
\draw (VII-b4)--(VII-a2);
\draw (VII-b4)--(VII-a3);
\end{scope}
\end{tikzpicture}
\caption{The seven ten-point posets, with incidence relations as specified
in \cite{DasMawiong2026}.}
\label{fig:ten-point-finite-spaces}
\end{figure}

For each of these posets and every
$m\geq2$, the space $P\oplus D_m$ is connected, has $m+10$ points, and satisfies
\[
 \bigl(\catw(P\oplus D_m),\cats(P\oplus D_m),\catLS(P\oplus D_m)\bigr)
 =(1,2,m).
\]
Moreover, $\catLS(\sd^k(P\oplus D_m))=2$ for every $k\geq1$.
We verify these assertions directly from the specified posets.

None of the seven posets has a beat point.  Indeed, each minimal point
has at least two minimal elements in its strict upper set, each maximal
point has at least two maximal elements in its strict lower set, and
both conditions hold for every middle point.  Thus every base is
noncontractible.

For Types I, III, IV, V, VI, and VII, the chosen points and the
beat-point reductions in Table~\ref{tab:ten-point-bases} show that deleting
$g$ leaves a contractible space.

\begin{table}[htb]
\centering
\begin{tabular}{ccl}
\toprule
Type & $g$ & Beat-point deletions after removing $g$ \\
\midrule
I & $a_2$ & $b_1,c_1,b_2,a_1,a_3,a_4,c_2,b_3$ \\
III & $c_1$ & $b_1,a_1,a_2,a_3,b_2,a_4,c_2,b_3$ \\
IV, V & $a_2$ & $b_1,b_2,c_1,b_3,a_1,a_3,c_2,b_4$ \\
VI, VII & $a_3$ & $b_3,b_4,c_1,b_1,a_1,a_2,c_2,b_2$ \\
\bottomrule
\end{tabular}
\caption{Essential weak points in six ten-point bases.  Each deletion
sequence leaves $c_3$.  Type II is treated separately using a cover of
its order complex by two strongly collapsible subcomplexes.}
\label{tab:ten-point-bases}
\end{table}

In Type III, the strict upper set of
$g=c_1$ has the tree with edges
\[
 b_1a_1,\ b_1a_2,\ b_1a_3,\ b_3a_3,\ b_3a_4
\]
as its Hasse graph.  For the other five types, the strict lower set of the
chosen $g$ has a path as its Hasse graph.  Thus every chosen $g$ is an
essential weak point, and Corollary~\ref{cor:essential-weak} applies.

For Type II we use a different cover of the order complex.  First,
$P$ is weakly contractible: delete $a_4$ and then $a_2$ as weak points.
Their strict lower sets, at the respective stages, have Hasse graphs
\[
 c_1-b_2-c_3-b_3-c_2,
 \qquad
 c_2-b_1-c_1-b_2-c_3.
\]
The remaining space admits the beat-point deletions
$b_2,c_1,b_1,a_1,a_3,c_2,c_3$, leaving $b_3$.
Consequently Proposition~\ref{prop:weak} gives $\catw(P\oplus D_m)=1$,
and Proposition~\ref{prop:ordinary} gives $\catLS(P\oplus D_m)=m$.

Put $K=\OO(P)$, and write $[u,v,w]$ for the simplex on the indicated
vertices, including all its faces.  Let $L_1$ be the subcomplex generated by
\[
 [c_1,b_1,a_3],\quad [c_2,b_1,a_3],\quad
 [c_1,b_2,a_4],\quad [c_3,b_2,a_4],
\]
and let $L_0$ be the subcomplex generated by the remaining fourteen
triangles of $K$.  Since $K$ is pure of dimension two, $K=L_0\cup L_1$.
Both subcomplexes are strongly collapsible, as follows.

In $L_1$, the vertices $a_3$ and $a_4$ are dominated by $b_1$ and $b_2$,
respectively.  Their deletion leaves the path
\[
 T=c_2-b_1-c_1-b_2-c_3,
\]
which is strongly collapsible.  In $L_0$, delete $a_3$ and $a_4$, each
dominated by $b_3$, and then delete $b_3$, dominated by $a_1$.
The remaining complex is the join of $T$ with the discrete two-vertex
complex on $\{a_1,a_2\}$.  This join is strongly collapsible, because
strong collapses are preserved by joins and the join of a vertex with
any complex is a cone.

It follows that
\[
 \OO(P\oplus D_m)
 =K*\OO(D_m)
 =\bigl(L_0*\OO(D_m)\bigr)\cup
   \bigl(L_1*\OO(D_m)\bigr)
\]
is covered by two strongly collapsible subcomplexes.  Hence
$\scat(\OO(P\oplus D_m))\leq2$, and
\eqref{eq:comparison-scat} gives $\catLS(\sd(P\oplus D_m))\leq2$.
By Lemma~\ref{lem:Xm-noncontractible} and
\eqref{eq:contractibility-detection}, no subdivision of $P\oplus D_m$
is contractible.  Monotonicity under subdivision therefore gives
$\catLS(\sd^k(P\oplus D_m))=2$ for every $k\geq1$, and
$\cats(P\oplus D_m)=2$.
\end{example}

\begin{example}[The Wallet]\label{ex:wallet}
The eleven-point Wallet $W$ of Barmak and Minian
\cite{BarmakMinian2008}(Figure \ref{fig:W})
is a noncontractible space without beat points.  They identify a weak
point $x$ for which $W\setminus\{x\}$ reduces to one point by beat-point
deletions.  Thus $x$ is essential in the terminology of
Definition~\ref{def:essential}, and $W\oplus D_m$ has category triple
$(1,2,m)$ on $m+11$ points, with category two after every positive number
of subdivisions.

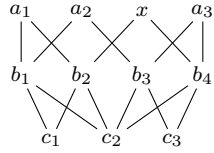
\begin{figure}[htbp]
\centering
\begin{tikzpicture}[x=0.8cm,y=0.85cm,
  every node/.style={font=\scriptsize,inner sep=1.5pt,fill=white}]

  \node (t1) at (0.0,2) {$a_1$};
  \node (t2) at (1.0,2) {$a_2$};
  \node (t3) at (2.0,2) {$x$};
  \node (t4) at (3.0,2) {$a_3$};

  \node (m1) at (0.0,1) {$b_1$};
  \node (m2) at (1.0,1) {$b_2$};
  \node (m3) at (2.0,1) {$b_3$};
  \node (m4) at (3.0,1) {$b_4$};

  \node (b1) at (0.5,0) {$c_1$};
  \node (b2) at (1.5,0) {$c_2$};
  \node (b3) at (2.5,0) {$c_3$};

  \draw (t1)--(m1);
  \draw (t1)--(m2);
  \draw (t2)--(m1);
  \draw (t2)--(m3);
  \draw (t3)--(m2);
  \draw (t3)--(m4);
  \draw (t4)--(m3);
  \draw (t4)--(m4);

  \draw (m1)--(b1);
  \draw (m1)--(b2);
  \draw (m2)--(b1);
  \draw (m2)--(b2);
  \draw (m3)--(b2);
  \draw (m3)--(b3);
  \draw (m4)--(b2);
  \draw (m4)--(b3);

\end{tikzpicture}
\caption{Weakly contractible and noncontractible eleven-point space.}
\label{fig:W}
\end{figure}
\end{example}

The lower bound of nine points for weakly contractible noncontractible
spaces \cite{CianciOttina2020} shows that the two nine-point bases are
optimal in size under the hypotheses of Theorem~\ref{thm:main}.
Consequently, twelve is the minimum size of a strict example obtained
from that theorem: the base needs at least nine points and strictness
requires $m\geq3$.  This is a minimum within the stated construction;
absolute minimality is a separate question.

\section{Additivity and realization}\label{sec:realization}

We next record that all three categories are additive under disjoint unions.
This is useful because the invariants use the unreduced convention.

\begin{proposition}\label{prop:additivity}
For nonempty finite $T_0$-spaces $X$ and $Y$,
\begin{align*}
 \catLS(X\sq Y)&=\catLS(X)+\catLS(Y),\\
 \cats(X\sq Y)&=\cats(X)+\cats(Y),\\
 \catw(X\sq Y)&=\catw(X)+\catw(Y).
\end{align*}
\end{proposition}

\begin{proof}
A categorical subset of $X\sq Y$ cannot meet both summands: the path
traced by a point during a homotopy stays within its summand.  Moreover,
a subset of $X$ is categorical in $X\sq Y$ exactly when it is categorical
in $X$, and similarly for $Y$.  Thus any categorical open cover splits
into covers of the two summands, giving the lower bound; adjoining optimal
covers gives the upper bound.  The same argument applies to classical category and
$|\OO(X\sq Y)|=|\OO(X)|\sq|\OO(Y)|$, proving the weak formula.

Barycentric subdivision commutes with disjoint union.  Hence for every $k$,
\[
 \catLS\bigl(\sd^k(X\sq Y)\bigr)
 =\catLS(\sd^kX)+\catLS(\sd^kY).
\]
The two nonincreasing integer sequences on the right are eventually
constant.  Taking $k$ beyond both stabilization indices proves the stable
formula.
\end{proof}

Besides the atoms supplied by Corollary~\ref{cor:subdivision-profile}, we use the height-one
family already noted in \cite[Remark~2.3]{CardenasFloresQuinteroVillar2025}.
Let $B_n$ be the poset with minima $u,v$, maxima $z_1,\dots,z_n$, and
relations
\[
  u<z_i,\qquad v<z_i\qquad(1\leq i\leq n).
\]

\Needspace{8\baselineskip}
\begin{proposition}\label{prop:Bn}
For every $n\geq2$,
\[
  \bigl(\catw(B_n),\cats(B_n),\catLS(B_n)\bigr)=(2,2,n).
\]
\end{proposition}

\begin{proof}
The order complex of $B_n$ is the connected graph $K_{2,n}$.  It contains a
cycle and is not contractible.  The dimension upper bound for category gives
$\catLS(|K_{2,n}|)\leq2$ \cite[Chapter~1]{CLOT}, and noncontractibility gives
the reverse inequality.
Thus $\catw(B_n)=2$.

The two subcomplexes $\operatorname{st}(u)$ and
$\operatorname{st}(v)$ cover $\OO(B_n)$ and are cones.  Therefore
$\scat(\OO(B_n))\leq2$, and \eqref{eq:comparison-scat} gives
$\cats(B_n)\leq2$.  The space $B_n$ has no beat points and is
noncontractible, so its stable category is not one.  Hence $\cats(B_n)=2$.

The $n$ principal opens $U_{z_i}=\{u,v,z_i\}$ are contractible and cover
$B_n$.  Conversely, suppose a categorical open subset contains two or more
maxima indexed by $J$.  It contains the open subspace
\[
  B_J=\{u,v\}\cup\{z_j:j\in J\}.
\]
The subspace $B_J$ has no beat points.  Sending all maxima outside $J$ to a
fixed $z_j$ defines a retraction $B_n\to B_J$, so the inclusion
$B_J\hookrightarrow B_n$ is not nullhomotopic.  Thus a categorical open set
contains at most one maximum, and $\catLS(B_n)=n$.
\end{proof}

Let $D_q$ denote a discrete space of $q$ points, with $D_0$ interpreted as
the empty summand.  For $q\geq1$ its category triple is $(q,q,q)$.

\begin{repeatedtheorem}{thm:realization}
\RealizationTheoremStatement
\end{repeatedtheorem}

\begin{proof}[Proof of Theorem~\ref{thm:realization}]
First suppose that $1\leq a<b\leq2a$ and $c\geq b$.  Put
\[
  d=b-a,\qquad q=2a-b,\qquad e=c-b.
\]
Then $d\geq1$ and $q,e\geq0$.  Consider
\begin{equation}\label{eq:realizing-space}
  Y=D_q\sq X_{e+2}\sq
    \bigsqcup_{j=1}^{d-1}X_2.
\end{equation}
By Corollary~\ref{cor:subdivision-profile} and Proposition~\ref{prop:additivity},
\begin{align*}
 \catw(Y)&=q+1+(d-1)=a,\\
 \cats(Y)&=q+2+2(d-1)=b,\\
 \catLS(Y)&=q+(e+2)+2(d-1)=c.
\end{align*}
The construction uses $c+9(b-a)$ points.

Now suppose $2\leq a\leq c$.  Since $c-a+2\geq2$, the space
\[
  D_{a-2}\sq B_{c-a+2}
\]
has category triple $(a,a,c)$ by Propositions~\ref{prop:additivity} and
\ref{prop:Bn}, and has $c+2$ points.  A one-point space realizes $(1,1,1)$.

Finally, if $\cats(X)=1$, then some subdivision of $X$ is contractible.
Contractibility detection under subdivision implies that $X$ itself is
contractible.  Hence $\catLS(X)=1$, proving that $(1,1,c)$ forces $c=1$.
\end{proof}

\begin{remark}
Allowing the empty sum, the additive monoid generated by the atoms
\[
  (1,1,1),\qquad (1,2,m)\quad(m\geq2)
\]
is exactly
\[
 \{(0,0,0)\}\cup\{(a,a,a):a\geq1\}\cup
 \{(a,b,c):1\leq a<b\leq2a,\ c\geq b\}.
\]
Indeed, the sums containing at least one atom $(1,2,m)$ give precisely the
strict region in the last set, while sums using only $(1,1,1)$ lie on the
displayed diagonal ray.  The family $B_n$ supplies the additional diagonal
atoms $(2,2,n)$ used in Theorem~\ref{thm:realization}.
\end{remark}

\section{Further questions}\label{sec:questions}

Our results lead to several natural problems.

\begin{conjecture}[Minimality of the strict example]
\label{conj:minimality}
Every finite $T_0$-space $X$ satisfying
$\catw(X)<\cats(X)<\catLS(X)$ has at least twelve points.  Equivalently,
$X_3$ has minimum cardinality among finite $T_0$-spaces on which both
inequalities are strict.
\end{conjecture}

Corollary~\ref{cor:subdivision-profile} gives an upper bound of twelve.  The classification of
Cianci and Ottina \cite{CianciOttina2020} shows that a weakly contractible,
noncontractible finite space has at least nine points, but this alone does not
settle the strict-separation minimum: a hypothetical smaller example need not
have weak category one or arise from our construction.  A proof of Conjecture~\ref{conj:minimality} must therefore also exclude
strict examples with $\catw(X)\geq2$.

\begin{question}
Which triples $(a,b,c)$ with $1\leq a\leq b\leq c$ are realized by connected
finite $T_0$-spaces?
\end{question}

Our additive constructions are generally disconnected.  The spaces $X_m$
show that connectedness imposes no upper bound on the ordinary category when
$(a,b)=(1,2)$.  It remains to find connected analogues for the wider region
of Theorem~\ref{thm:realization}.

\begin{question}
Does every finite $T_0$-space satisfy
$\cats(X)\leq2\catw(X)$?
\end{question}

The spaces $X_m$ attain equality.  An affirmative answer, together with
Theorem~\ref{thm:realization}, would determine the full set $\mathfrak R$:
it would consist exactly of the regions in that theorem.  A negative answer
would produce triples outside the range of the present additive construction.

\section*{Acknowledgements}
ChatGPT (OpenAI) was used to assist with language editing and presentation.

\end{document}